\documentclass[11pt]{article}
\usepackage{graphicx}
\usepackage{amsmath, amsthm}

\newfont{\bb}{msbm10 at 11pt}
\def\r{\hbox{\bb R}}
\def\e{\hbox{\bf E}}
\def\h{\hbox{\bb H}}

\theoremstyle{plain}
\newtheorem{theorem}{Theorem}[section]

\newtheorem{corollary}[theorem]{Corollary}

\theoremstyle{definition}
\newtheorem{definition}[theorem]{Definition}

\date{}
\begin{document}
\title{\textbf{$k-$type partially null and pseudo null slant  helices in Minkowski 4-space}}
\author{Ahmad T. Ali, Rafael L\'{o}pez\footnote{Partially supported by MEC-FEDER
 grant no. MTM2007-61775 and
Junta de Andaluc\'{\i}a grant no. P06-FQM-01642.}\ \ and Melih Turgut\footnote{Corresponding author.}}

\maketitle

\begin{abstract}
We introduce the notion of $k$-type slant helix in Minkowski space $\e_1^4$. For partially null and pseudo null curves in $\e_1^4$, we express some characterizations in terms of their curvature and torsion functions.
\end{abstract}


\section{Introduction}
The notion of a slant helix is due to Izumiya and Takeuchi \cite{izu}. A curve $\alpha$ with non-vanishing curvature is called a slant helix in Euclidean space $\e^{3}$ if the principal normal lines of $\alpha$ make a constant angle with a fixed direction of the ambient space. Later, spherical images, the tangent and the binormal indicatrix and some characterizations of such curves were presented in \cite{ky}. Recently,  further characterizations and position vectors of such curves are given in \cite{ali3,keyi,ty1}.

In recent years, by the coming of theory of relativity, researchers extended some topics of classical differential geometry  to Lorentzian manifolds and there exists an extensive literature on this subject. For instance, about general helices on Lorentzian geometry, we refer  \cite{al1,al2,ey,fgl, ib1, ib2,kk,ty2}.

Now we focus in Minkowski $4$-dimensional space $\e_1^4$, that is, the real vector $4$-dimensional space $\r^4$ equipped with
the standard flat metric given by
$$g=-dx_{1}^{2}+dx_{2}^{2}+dx_{3}^{2}+dx_{4}^{2},$$
where $(x_1,x_2,x_3,x_4)$ is a rectangular coordinate system in $\r^4$. A spacelike curve $\alpha:I\subset\r\rightarrow \e_1^4$ is called spacelike if the induced metric is Riemannian. For spacelike curves parameterized by the length-arc, one has defined a Frenet frame $\{V_1,\ldots,V_4\}$, where $V_1(s)=\alpha'(s)$. In Minkowski $4$-dimensional space $\e_1^4$, we extend the concept of slant helix as follows:

\begin{definition} Let $\alpha:I\rightarrow \e_1^4$ be a spacelike curve, with Frenet frame $\{V_1,V_2,V_3,V_4\}$. We say that $\alpha$ is a $k$-type slant helix if there exists a (non-zero) constant vector field $U\in \e_1^4$ such that
$g(V_{k+1},U)$ is constant, for $0\leq k\leq 3$. The vector $U$ is called an \emph{axis} of the curve.
 \end{definition}
In particular, $0$-type slant helices are general helices and $1$-type slant helices are slant helices. In this work we  consider $k$-type slant helices for partially null curves and pseudo null curves. Recall that a  partially null curve is a (spacelike) curve where $V_2$ is spacelike and $V_3$ is a lightlike vector. On the other hand, a pseudo null curve is a (spacelike) curve
if $V_2$ is a lightlike vector. In all cases, we characterize $k$-type slant helices in terms of the curvatures of the curve and we determine the axis of the curve. Next we focus on $k$-type pseudo null slant helices in hyperbolic space. Finally, we point out that type-$3$ slant helices in $\e_1^4$ have studied in \cite{ty2} for those curves where the Frenet frame are non-lightlike vectors.

\section{Preliminaries}

The Lorentzian metric $g$ in  Minkowski space $\e_1^4$ is indefinite. Therefore a vector $v\in\e_1^4$ can
have one of the three causal characters. We say that $v$ is    spacelike if $g(v,v)>0$
or $v=0$, timelike if $g(v,v)<0$ and lightlike (or null) if $g(v,v)=0$ and $v\neq 0$. Similarly, an arbitrary curve $\alpha=\alpha(s)$ in $\e_1^4$
is called spacelike, timelike or lightlike, if all of its
velocity vectors $\alpha'(s)$ are  spacelike,
timelike or lightlike, respectively. The norm of a vector $v\in\e_1^4$ is given by
$\left\Vert v\right\Vert=\sqrt{\left\vert g(v,v)\right\vert}$. Therefore, $v$ is a unit vector if $g(v,v)=\pm1$.
A (spacelike or timelike) curve is parametrized by the   arclength  is $\alpha'(s)$ is a unit vector for any $s$. Also, we say that the vectors $v,w$ in $\e_1^4$ are   orthogonal if $g(v,w)=0$.

Consider $\alpha=\alpha(s)$ a spacelike curve where $s$ is the length-arc parameter. Denote by $\left\{T(s),N(s),B_1(s),B_2(s)\right\}$ the moving Frenet frame along the curve $\alpha(s)$. Then $T,N,B_1,B_2$ are
called the tangent, the principal normal, the first binormal and the second binormal  vector fields of $\alpha$, respectively. The fact that the metric $g$ is indefinite causes that the vector $N$, $B_1$ and $B_2$ have different causal characters (here $T$ is a spacelike vector, since $\alpha$ is a spacelike curve).

In this paper we are interesting for partially null curves and pseudo null curves (see \cite{wal}). A \emph{partially null curve} is a spacelike curve where
$N$ is spacelike and $B_1$ is lightlike. In such case, the vector $B_2$ is the unique lightlike vector orthogonal to $T$ and $N$ such that $g( B_1,B_2)=1$. The Frenet equations are
given by (\cite{cis, wal}):
\begin{equation}\label{1}
\left[
\begin{array}{c}
T' \\
N' \\
B'_1\\
B'_2
\end{array}\right]=
\left[
\begin{array}{cccc}
0 & \kappa  & 0 & 0 \\
-\kappa  & 0 & \tau  & 0 \\
0 & 0 & \sigma  & 0 \\
0 & -\tau  & 0 & -\sigma
\end{array}%
\right] \left[
\begin{array}{c}
T \\
N \\
B_1 \\
B_2
\end{array}
\right].
\end{equation}
Here, $\kappa,\,\tau$ and $\sigma$ are first, second and third
curvature of the curve $\alpha$, respectively. In fact, one can prove that after a null rotation of the ambient space, the curvature $\sigma$ can be chosen to be zero (and $\tau$ is determined up to a constant). This means that any partially null curve lies in a three dimensional lightlike subespace (orthogonal to $B_1$).

A spacelike curve $\alpha(s)$ is called a \emph{pseudo null curve} if $\alpha''(s)$ is a lightlike vector for any $s$. Then the normal vector   is $N=T'$. If $N'$ is lightlike, then $\alpha$ is included in a lightlike plane: we discard this trivial case. In the rest of cases, $B_1$ is a unit spacelike vector orthogonal to $\{T,N\}$ and $B_2$ is the unique lightlike vector orthogonal to
 $T$ and $B_1$ such that $g(N,B_2)=1$. The Frenet equations are (\cite{cis, wal}):
\begin{equation}\label{2}
\left[
\begin{array}{c}
T' \\
N' \\
B'_1\\
B'_2
\end{array}\right]=
\left[
\begin{array}{cccc}
0 & \kappa  & 0 & 0 \\
0 & 0 & \tau  & 0 \\
0 & \sigma  & 0 & -\tau  \\
-\kappa  & 0 & -\sigma  & 0%
\end{array}%
\right] \left[
\begin{array}{c}
T \\
N \\
B_1 \\
B_2
\end{array}
\right]
\end{equation}
Here  the first curvature $\kappa $ can take only two values: $\kappa\equiv 0$ when the
curve is a straight-line or $\kappa\equiv 1$ in all other cases. We will assume non-trivial cases, that is, that $\kappa\equiv 1$, as well as, $\sigma,\tau\not=0$. The classification of $k$-type slant helices that we present in this paper corresponds to partially null curves (Section \ref{sect-3}) and pseudo null curves (Section \ref{sect-4}). In the last setting, we also consider pseudo null curves in pseudo-hyperbolic space.

We end this section about some of notation. If $U$ is an axis of an $k$-type slant helix, we decompose $U$ with respect to the Frenet frame $\{T,N,B_1,B_2\}$ as
$U=u_1T+u_2 N+u_3 B_1+u_4 B_2$, where $u_i=u_i(s)$ are differentiable functions on $s$.
\section{$k$-type  partially null slant helices}\label{sect-3}

In this section we study  $k$-type partially null slant helices. Recall that in the Frenet equations (\ref{1}), the value of $\sigma$ is zero. We will assume that $\kappa,\tau\not=0$. In particular, $B_1$ is a constant vector. If we take  $U=B_1$,  then, $g(T,U)=g(N,U)=(B_1,U)=0$ and $g(B_2,U)=1$. This means that any partially null curve is a $k$-type slant helix for any $k\in\{0,\ldots,3\}$. In this section, we discard the trivial case that  $U=B_1$.

\begin{theorem}\label{th1}
Let $\alpha$ be a partially null curve in  $\e_1^4$. Then $\alpha$ is
 a $0$-type slant helix (or general helix) if and only if
\begin{equation}\label{7}
\dfrac{\tau}{\kappa}=constant.
\end{equation}
Moreover,   $\alpha$ is also a $k$-type slant helix, for $k\in\{1,2,3\}$.
\end{theorem}

\begin{proof}
Assume that $\alpha$ is  a $0$-type slant helix. Then for a constant vector field $U$, the function $g(T,U)=c$ is constant.  Differentiating this equation and  from Frenet equations, we obtain  that $\kappa g( N,U)=0$. So $U$ is orthogonal to $N$
and we decompose $U$ as
\begin{equation}\label{9}
U=c\,T+u_3\,B_1+u_4\,B_2.
\end{equation}
  Differentiating (\ref{9}) and using the Frenet equations (\ref{1}), one   arrives to
$$
\left\{
\begin{array}{l}
c\kappa-\tau\,u_4=0,\\
u'_3=0,\\
u'_4=0.
\end{array}
\right.
$$
Thus $u_3$ and $u_4$ are constant. From the first equation, we obtain  (\ref{7}).

Conversely, suppose that the relation (\ref{7}) holds. We define the vector field
\begin{equation}\label{11}
U=\frac{\tau}{\kappa}T+B_2.
\end{equation}
Differentiating (\ref{11}) and considering (\ref{7}) and (\ref{1}), we have $U'=0$, that is, $U$ is a constant vector field.
Because $g(T,U)=\tau/\kappa$ is constant, we have that $\alpha$ is a $0$-type slant helix.

Finally, and because the coefficients of $U$ are constant, we have  $g(V_{k+1},U)=0$,  for any
$0\leq k\leq 3$, that is, $\alpha$ is a $k$-type slant helix.
\end{proof}

 From the above theorem, we have

\begin{corollary}\label{lm1}
If $\alpha$ is a $0$-type  partially null slant helix in  $\e_1^4$, an axis of
$\alpha$ is
$$D=\frac{\tau}{\kappa}T+B_1+B_2.$$
\end{corollary}

As a particular case of $0$-type partially null slant helices are those  with $\kappa$ and $\tau$ constant. In such case,
 $\alpha$ parametrizes as (see \cite{bo})
$$\alpha(s)=(cs,\frac{1}{\kappa}\cos(\kappa s),\frac{1}{\kappa}\sin(\kappa s),cs).$$

\begin{theorem}\label{th3}
Let $\alpha$ be a partially null curve in  $\e_1^4$. Then $\alpha$ is $1$-type slant helix if and only if, there exists a constant $C$ such that
\begin{equation}\label{16}
\dfrac{\tau(s)}{\kappa(s)}-C\int_{0}^{s}\kappa(t)\,dt=0.
\end{equation}
Moreover, $\alpha$ is a $2$-type slant helix.
\end{theorem}
\begin{proof}
Assume that  $\alpha$ is a $1$-type slant helix. There exists a  constant vector field $U$ such that
$g(N,U)=c$ is constant. We decompose $U$ as $U=u_1 T+c N+u_3 B_1+u_4 B_2$. Differentiation this equation, and using (\ref{1}), we have the following system of ordinary differential equations%
\begin{equation}\label{18}
\left\{
\begin{array}{l}
u'_1-c \kappa =0,\\
\kappa\,u_1-\tau\,u_4=0,\\
u'_3+c \tau =0,\\
u'_4=0.
\end{array}
\right.
\end{equation}
In particular, $u_4$ is constant. Then we obtain
\begin{equation}\label{19}
\left\{
\begin{array}{l}
u_{1}=c\int_{0}^{s}\kappa(t) dt,\\
u_{3}=-c\int_{0}^{s}\tau(t) dt,\\
\end{array}
\right.
\end{equation}
Using the second equation in (\ref{18}) together (\ref{19}), and letting $C=\dfrac{c}{u_{4}}$ we conclude (\ref{16}).
Conversely, assume that  the relation (\ref{16}) holds. Then, let us consider the following vector
$$U=(\int_{0}^{s}\kappa(t)
dt)\,T+N-\frac{1}{C}(\int_{0}^{s}\tau(t) dt)B_1+B_{2},$$
for the real number $C$ given in (\ref{16}). Differentiating  $U$ and using the Frenet equations (\ref{1}), we have $U'=0$, that is, $U$ is a constant vector field. Moreover, $g(N,U)=1$
which means that $\alpha$ is a $1$-type  slant helix.

Because the coefficient $u_4$ is constant, then $g(B_1,U)$ is constant and then $\alpha$ is a $2$-type slant helix.
\end{proof}

As a consequence  of the  above theorem, we have

\begin{corollary}\label{lm3}
Let $\alpha$ be a partially null curve in $\e_1^4$. If $\alpha$ is a $1$-type   slant helix, then $$
D=(\int\limits_{0}^{s}\kappa(t) dt)\,T+N-\frac{1}{C}(\int\limits_{0}^{s}\tau(t) dt)B_1+B_2,
$$
is a constant vector.
\end{corollary}

We now consider $2$-type slant helices.

\begin{theorem}\label{th5}
Let $\alpha$ be a $2$-type partially null slant helix. Then an axis of $\alpha$ is
\begin{equation}\label{18-12-1}
U=v_1\,T+\frac{v'_1}{\kappa}\,N+\Big(c_3-\int\frac{\tau}{\kappa}v'_1\,ds\Big)\,B_1+B_2,
\end{equation}
where
\begin{equation}\label{18-12-2}
\begin{array}{l}
v_1(s)=\cos\Big[\int\kappa\,ds\Big]\Big(c_1-\int\tau\sin\Big[\int\kappa\,ds\Big]ds\Big)\\
\,\,\,\,\,\,\,\,\,\,\,\,\,\,\,\,\,\,\,\,+\sin\Big[\int\kappa\,ds\Big]\Big(c_2+\int\tau\cos\Big[\int\kappa\,ds\Big]ds\Big).
\end{array}
\end{equation}
\end{theorem}

\begin{proof}
We know that there exists a constant vector field $U$ such that $g(B_1,U)=c$, $c$ is a constant.
We decompose $U$ as
$$
U=u_1\,T+u_2\,N+u_3\,B_1+c\,B_2.
$$
By differentiating $U$  we have the following system
of ordinary differential equations
\begin{equation}\label{18-11}
\left\{
\begin{array}{l}
u'_1-\kappa\,u_2=0,\\
u'_2+\kappa\,u_1-c\,\tau =0,\\
u'_3+\tau\,u_2=0,
\end{array}
\right.
\end{equation}
From the first equation, we have
\begin{equation}\label{18-12}
u_2=\frac{u'_1}{\kappa}=u'(\theta),
\end{equation}
where $\theta=\int\kappa(s)ds$ is a new variable. Substituting (\ref{18-12}) in the second equation of the system (\ref{18-11}), we have the following ordinary differential equation
$$u''_1(\theta)+u_1(\theta)=c\,f(\theta),
$$where $f(\theta)=\frac{\tau}{\kappa}(\theta)$. By solving the above equation we obtain
$$
\left\{
\begin{array}{l}
u_1=c\Big[\cos\theta\Big(c_1-\int f(\theta)\sin\theta\,d\theta\Big)+
\sin\theta\Big(c_2+\int f(\theta)\cos\theta\,d\theta\Big)\Big],\\
u_2=c\Big[\cos\theta\Big(c_2+\int f(\theta)\cos\theta\,d\theta\Big)-
\sin\theta\Big(c_1-\int f(\theta)\sin\theta\,d\theta\Big)\Big],\\
u_3=c\Bigg[c_3-\int\Big[\cos\theta\Big(c_2+\int f(\theta)\cos\theta\,d\theta\Big)-
\sin\theta\Big(c_1-\int f(\theta)\sin\theta\,d\theta\Big)\Big]\,d\theta\Bigg].
\end{array}
\right.
$$
This gives (\ref{18-12-1}) and (\ref{18-12-2}).

\end{proof}

Now we study $3$-type  partially null slant helices. Recall that for partially null curves, the vector $B_1$ is constant.

\begin{theorem}\label{th7}
Let $\alpha$ be a partially null curve in $\e_1^4$. If $\alpha$ is a $3$-type slant helix,
 then $\alpha$ is a $k$-type slant helix, for $k\in\{0,1,2\}$.
\end{theorem}
\begin{proof}
Let $U$ be the constant vector field such that $g(B_2,U)=c$ is constant. A differentiation of $g(B_2,U)=c$ leads to $\tau\,g(N,U)=0$, that is, $g(N,U)=0$.  We write $U$ as $U=u_1 T+c B_1+u_4 B_2$. We differentiate $U$ obtaining
$$
\left\{
\begin{array}{l}
u'_1=0,\\
u_1\kappa-\tau u_4=0,\\
u_4'=0.
\end{array}
\right.
$$
This implies that $u_1$ and $u_4$ are constant, and from the second equation, the quotient $\tau/\kappa$ is constant too. By using Theorem \ref{th1}, $\alpha$ is a $0$-type slant helix.

Moreover, the coefficients of $U$ with respect to the Frenet frame are all constant, which means that $\alpha$ is a $k$-type slant helix for $0\leq k\leq 3$.
\end{proof}
As a consequence of the proof of Theorems \ref{th1}, \ref{th3} and \ref{th7}, we obtain the following

\begin{corollary} Let $\alpha$ be a partially null curve in $\e_1^4$. Then $\alpha$ is a $0$-type slant helix if and only if $\alpha$ is a $3$-type slant helix.
\end{corollary}

\begin{corollary} Let $\alpha$ be a partially null curve in $\e_1^4$. Assume that there exists a constant  vector field $U$ such that $g(V_{k+1},U)=0$ for some $k\in\{0,1,3\}$. Then $\alpha$ is a $k$-type slant helix for any $k\in\{0,1,3\}$.
\end{corollary}

\begin{corollary} Let $\alpha$ be a partially null curve in $\e_1^4$. Assume that the curvatures $\kappa$ and $\tau$ are constant. Then $\alpha$ is a $k$-type slant helix for $k\in\{0,1,3\}$.
\end{corollary}

As consequence of our study, we show a set of relations between $k$-type partially null slant curves:

\begin{figure}[hbtp]
\begin{center}\includegraphics[width=7cm]{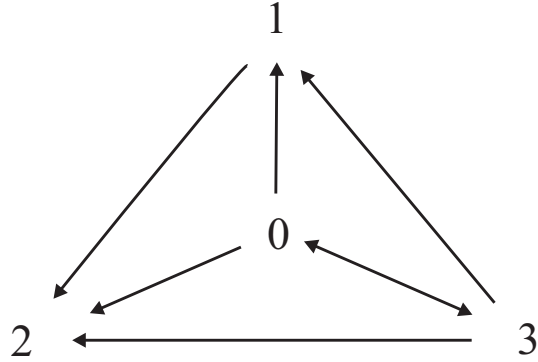}\end{center}
\caption{The arrows mean the implication between the different of $k$-type partially null slant helices}
\end{figure}
\section{$k$-type  pseudo  null slant helices}\label{sect-4}

In this section we study pseudo null curves that are  $k$-type slant helices. Recall that we assume $\kappa\equiv 1$ in (\ref{2}).
We begin for general helices.

\begin{theorem}\label{th2}
There does not exist $0$-type  pseudo null slant helices   in $\e_1^4$.
\end{theorem}
\begin{proof}
If such curve does exist, there exists a constant vector field $U$ such that $g(T,U)=c$ is constant. Differentiating this equation and using the Frenet equations (\ref{2}), we obtain, $g(N,U)=0$. Then the decomposition of
$U$ in terms of the Frenet frame is $U=c\,T+u_2\,N+u_3\,B_1$. As $U'=0$, by using the Frenet equations (\ref{2}), we have
$$
\left\{
\begin{array}{l}
u'_2+\sigma\,u_3+c=0,\\
u'_3+\tau\,u_2=0,\\
\tau\,u_3=0\\
\end{array}
\right.
$$
From the third equation, $u_3=0$, and thus, $u_2=c=0$, that is, $U=0$: contradiction.
\end{proof}

We characterize $1$-type  pseudo null slant helices as follows:

\begin{theorem}\label{th4}
Let $\alpha$ be a pseudo null curve in  $\e_1^4$. Then $\alpha$ is  $1$-type slant helix   if and only if
\begin{equation}\label{19-5}
\frac{\sigma(s)}{\tau(s)}=-\frac{s^2}{2}+a s+b.
\end{equation}
for some constants $a,b\in\r$. Moreover $\alpha$ is also a $2$-type slant helix.
\end{theorem}

\begin{proof}
Assume that $\alpha$ is a $1$-type slant helix and let $U$ be a constant vector field such that
$g(N,U)=c$ is constant. Differentiation this equation and by Frenet equations, $\tau g(B_1,U)=0$, and so,
$g(B_1,U)=0$. If we write $U$ as linear combination of the Frenet frame, we have
$U=u_1 T+u_2 N+c B_2$. Differentiating $U$ we obtain the following system of ordinary differential equations
$$\left\{
\begin{array}{l}
u'_1-c=0,\\
u'_2+u_1=0,\\
\tau\,u_2-c\sigma=0.
\end{array}
\right.
$$
where $u_i$ are the coefficients of $U$ in the decomposition with respect to the Frenet frame. If $c=0$, $u_1=u_2=0$, that is, $U=0$: contradiction. Thus, $c\not=0$.
From the last equation, $u_2=c\Big(\frac{\sigma}{\tau}\Big)$ and from the first two equations,
$u_2$ satisfies $u_2''+c=0$. The solution of this equation is
$$u_2=-c\frac{s^2}{2}+\lambda s+\mu,\ \ \lambda,\mu\in\r.$$
Thus $\sigma/\tau$ satisfies  the condition (\ref{19-5}), with $a=\lambda/c$ and $b=\mu/c$.

Conversely, let us assume that  the relation (\ref{19-5}) holds. Define  the following vector field
\begin{equation}\label{20}
U=-\Big(\frac{\sigma}{\tau}\Big)'\,T+\Big(\frac{\sigma}{\tau}\Big)\,N+B_2.
\end{equation}
Differentiating  (\ref{20}), and using (\ref{2}), we easily have $U'=0$, that is, $U$ is a constant vector. Moreover,
$g(N,U)=1$, showing that $\alpha$ is a $1$-type slant helix. Finally, as $g(B_1,U)=0$, then $\alpha$ is a $2$-type slant helix.
\end{proof}

\begin{corollary}\label{lm4}
Let $\alpha$ be a $1$-type  pseudo null slant helix. Then an axis of $\alpha$ is
$$
D=-\Big(\frac{\sigma}{\tau}\Big)'\,T+\Big(\frac{\sigma}{\tau}\Big)\,N+B_2.
$$
\end{corollary}

\begin{theorem}\label{th6}
Let $\alpha$ be a pseudo null curve in $\e_1^4$. Then $\alpha$ is a $2$-type slant helix  if and only if
\begin{equation}\label{30}
\int\tau ds+\dfrac{d}{ds}\left[\sigma +\dfrac{d}{ds}\left(
\dfrac{\sigma }{\tau }\int\tau ds\right)\right]=0.
\end{equation}
\end{theorem}

\begin{proof}
If $\alpha$ is a $2$-type slant helix, there exists a constant vector $U$ such that $g(B_1,U)=c$ is constant.
  We write $U$ as $U=u_1T+u_2N+cB_1+u_4 B_2$. Differentiating $U$ we obtain
\begin{equation}\label{32}
\left\{
\begin{array}{l}
u'_1-u_4=0,\\
u'_2+u_1+c\sigma =0,\\
\tau\,u_2-\sigma u_4=0,\\
u'_4-c\tau =0.
\end{array}
\right.
\end{equation}
Remark that $c\not=0$. By solving this system, we conclude from the last equation that
$$
\left\{
\begin{array}{l}
u_1=-c\left[\sigma+\dfrac{d}{ds}\left(\dfrac{\sigma}{\tau}\int\tau\,ds\right)\right],\\
u_2=c\dfrac{\sigma}{\tau}\,\int\tau\,ds,\\
u_4=c\int\tau\,ds.
\end{array}
\right.
$$
Hence together the first two equations of(\ref{32}), we conclude  (\ref{30}).

Conversely, assume that (\ref{30}) holds. Define
\begin{equation}\label{34}
U= -\left[\sigma+\dfrac{d}{ds}\left(\dfrac{\sigma\,\int\tau\ ds}{\tau}\right)\right]\,T+\left(\dfrac{\sigma\,
\int\tau\ ds}{\tau}\right)\,N+B_1+ \left[\int\tau\,ds\right]\,B_2.
\end{equation}
 Considering (\ref{30}) and
differentiating (\ref{34}), we have by using (\ref{2}), that $U'=0$, that is, $U$ is constant. Moreover, $g(B_{1},U)=1$, which shows that $\alpha$ is a $2$-type slant helix.
\end{proof}

\begin{corollary}\label{lm6}
An axis of a $2$-type  pseudo null slant helix is the vector given by
$$
D=-\left[\sigma+\dfrac{d}{ds}\left(\dfrac{\sigma\,\int\tau\ ds}{\tau}\right)\right]\,T+\left(\dfrac{\sigma\,
\int\tau\ ds}{\tau}\right)\,N+B_1+ \left[\int\tau\,ds\right]\,B_2.
$$
\end{corollary}

\begin{theorem}\label{th8}
Let $\alpha$ be a pseudo null curve in $\e_1^4$. Then $\alpha$ is $3$-type slant helix   if and only if
\begin{equation}\label{36}
\frac{\tau}{\sqrt{1+\sigma^2}}+\frac{d}{ds}
\Big[\frac{\sqrt{1+\sigma^2}\Big(\sigma\tau'(1+\sigma^2)+\tau\sigma'(2-\sigma^2)
\Big)}{\tau(1+\sigma^2)^2-3\tau\tau'^2+\sigma''(1+\sigma^2)}\Big]=0.
\end{equation}
\end{theorem}

\begin{proof}
Suppose that  $\alpha$ is a $3$-type slant helix and let  $U$ be the constant vector field such that
$g(B_2,U)=c$ is constant. If we write $U=u_1\,T+c N+u_3\,B_1+u_4\,B_2$, a differentiation of this expression leads to \begin{equation}\label{18-1-1}
\left\{
\begin{array}{l}
u'_1-u_4=0,\\
u_1+\sigma\,u_3=0,\\
u'_3+c\tau -\sigma\,u_4=0,\\
u'_4-\tau\,u_3=0.
\end{array}
\right.
\end{equation}
The first and second equation of (\ref{18-1-1}) give
$$
\left\{
\begin{array}{l}
u_4=u'_1,\\
u_3=-\frac{u_1}{\sigma},
\end{array}
\right.
$$
Substituting $u_3$ and $u_4$ into the two last equations of (\ref{18-1-1}), we have
$$\left\{
\begin{array}{l}
u'_1=\frac{
\sigma'}{\sigma(1+\sigma^2)}u_1+\frac{\tau\sigma}{1+\sigma^2}c,\\
u''_1+\frac{\tau}{\sigma}u_1=0.
\end{array}
\right.
$$
By the last two expressions of $u_1$, we obtain (\ref{36}). From the above system of equations, we have the following two equations
$$u_1=\dfrac{c\sigma\Big(\sigma\tau'(1+\sigma^2)+\tau\sigma'(2-\sigma^2)
\Big)}{\tau(1+\sigma^2)^2-3\sigma\sigma'+\sigma''(1+\sigma^2)}.$$
and
$$u_1=\frac{c\sigma}{\sqrt{1+\sigma^2}}\int\frac{\tau}{\sqrt{1+\sigma^2}}ds.$$

Conversely, assume that (\ref{36}) holds. Then we define the following vector
\begin{equation}\label{20-1}
U=\frac{\sigma}{\sqrt{1+\sigma^2}}\epsilon\,T
-\frac{1}{\sqrt{1+\sigma^2}}\epsilon\,B_1+
\Big(\tau\sigma-\frac{\sigma'}{\sqrt{1+\sigma^2}}\epsilon\Big)u_2\,B_2,
\end{equation}
where $\epsilon=\int\frac{\tau}{\sqrt{1+\sigma^2}}ds$. Differentiating of (\ref{20-1}), we easily have $U'=0$. Moreover one  obtains $g(B_2,U)=0$ which means that $\alpha$ is a $3$-type slant helix.
\end{proof}

From the above theorem, one concludes

\begin{corollary}\label{lm8}
If  $\alpha$  is a $3$-type  pseudo null slant helix in $\e_1^4$, an axis of $\alpha$ is
the vector
$$
D=\frac{\sigma}{\sqrt{1+\sigma^2}}\epsilon\,T
-\frac{1}{\sqrt{1+\sigma^2}}\epsilon\,B_1+
\Big(\tau\sigma-\frac{\sigma'}{\sqrt{1+\sigma^2}}\epsilon\Big)u_2\,B_2,
$$
where $\epsilon=\int\frac{\tau}{\sqrt{1+\sigma^2}}ds$.
\end{corollary}

We end this section focusing into $k$-type  pseudo null slant helices in the pseudohyperbolic space. Recall that
 the  pseudohyperbolic space of $\e_1^4$ of radius $r$ and centered at $x_0\in\e_1^4$ is defines by $\h_0^3=\{x\in\e_1^4;g(x-x_0,x-x_0=-r^2\}$. The metric $g$ induced into $\h_0^3$ is Riemannian with constant negative intrinsic curvature.   In \cite{cis} it is proved that a pseudo null curve $\alpha$ lies in
$\h_0^3$ if and only if the quotient $\sigma/\tau$ is a negative constant. We particularize our results for pseudo null curves of $\h_0^3$.  As consequence of  Theorem \ref{th4}, we have

\begin{corollary}\label{lm4-1}
There does not exist a  $1$-type  pseudo null slant
helix in pseudohyperbolic spaces of $\e_1^4$.
\end{corollary}

As a consequence of Theorem \ref{th6} we have

\begin{corollary}\label{lm6-1}
If $\alpha$ is a $2$-type  pseudo null slant helix in $\h_0^3$, then
\begin{equation}\label{35}
\tau=\lambda e^{\frac{s}{\sqrt{-2c}}}+\mu   e^{-\frac{s}{\sqrt{-2c}}},\ \ \lambda,\mu\in\r
\end{equation}
where $c$ is the negative constant given by $\sigma=c \tau$.
\end{corollary}

\begin{proof} From (\cite{cis}) we know that $\sigma=c\tau$ for some real number $c<0$. If we put this into (\ref{30}), we have that $\tau$ satisfies $2c\tau''+\tau=0$, whose solutions are given by (\ref{35}).
\end{proof}

If we put $\sigma=c\tau$ for pseudo null curves in $\h_0^3$, Theorem \ref{th8} implies
\begin{corollary}\label{lm8-1}
If $\alpha$ is a $3$-type  pseudo null slant helix in $\h_0^3$, the second curvature function $\tau(s)$ satisfies
$$2c^2\tau\tau'\tau'''=
c\tau''\Big[5\tau^2(1+c^2\tau^2)+c(3\tau'^2+4\tau\tau'')\Big]+c^2\tau^5(2+c^2\tau^2)+\tau^3(1-15 c^3\tau'^2).$$
\end{corollary}

{\bf Acknowledgements.} The third author would like to thank T\"{u}bitak-Bideb for their financial support during his Ph.D. studies.


\vspace{5mm}
\author{\textbf{Ahmad T. Ali}:

\textbf{Present address:} King Abdul aziz University, Faculty of Science, Department of Mathematics, PO Box 80203, Jeddah, 21589, Saudi Arabia.

\textbf{Permanent address:} Mathematics Department, Faculty of Science, Al-Azhar University, Nasr City, 11448, Cairo, Egypt.}

{\em e-mail address:} {\tt atali71@yahoo.com} and {\tt aabdullwhaab@kau.edu.sa}

\vspace{5mm}

\author{\textbf{Rafael L\'opez}:

Departamento de Geometr\'{\i}a y Topolog\'{\i}a, Universidad de Granada, 18071 Granada, Spain.}

{\em e-mail address:} {\tt rcamino@ugr.es}

\vspace{5mm}
\author{\textbf{Melih Turgut}:

Dokuz Eyl\"{u}l University, Buca Educational Faculty, Department of Mathematics, 35160, Buca-Izmir, Turkey}

{\em e-mail address:} {\tt Melih.Turgut@gmail.com}

\end{document}